\documentclass{amsart}
\usepackage[utf8]{inputenc}
\usepackage{tikz}
\usepackage{tikz-cd}
\usepackage{amssymb}

\usepackage{fullpage}

\newcommand{\R}{\mathbb{R}}
\newcommand{\Z}{\mathbb{Z}}

\newcommand{\F}{\mathbb{F}_7}
\newcommand{\op}{\operatorname}
\newcommand{\Map}{\op{Map}}

\newtheorem{Thm}{Theorem}
\newtheorem{Ex}[Thm]{Example}
\newtheorem{Rem}[Thm]{Remark}
\newtheorem{Prop}[Thm]{Proposition}
\newtheorem{Lem}[Thm]{Lemma}

\title{The string coproduct "knows" Reidemeister/Whitehead torsion}
\author{Florian Naef}
\address{Department of Mathematical Sciences, University of Copenhagen, Copenhagen, Denmark}
\email{flna@math.ku.dk}

\begin{document}

\maketitle

\begin{abstract}
We show that the string coproduct is not homotopy invariant. More precisely, we show that the (reduced) coproducts are different on $L(1,7)$ and $L(2,7)$. Moreover, the coproduct on $L(k,7)$ can be expressed in terms of the Reidemeister torsion and hence transforms with respect to the Whitehead torsion of a homotopy equivalence. The string coproduct can thereby be used to compute the image of the Whitehead torsion under the Dennis trace map.
\end{abstract}

\section{Introduction}
Given a compact oriented manifold $M$ of dimension $n$, Chas and Sullivan define a number of operations on the homology of the free loop space $LM = \Map(S^1, M)$ \cite{ChasSullivan, ChasSullivan2}. The most prominent ones are the string product, which is an operation of the type
$$
\star \colon H_\bullet(LM \times LM) \to H_{\bullet - n}(LM),
$$
the string coproduct
$$
\Delta \colon H_\bullet(LM, M) \to H_{\bullet - n + 1}(LM \times LM, M \times LM \cup LM \times M),
$$
and the circle action
$$
B \colon H_\bullet(LM) \to H_{\bullet + 1}(LM).
$$
The string product and coproduct are defined in terms of intersections of chains satisfying certain transversality condition. In particular, it is not a priori clear whether or not they depend on the manifold structure beyond its homotopy type. Or said differently, one can ask whether a homotopy equivalence $f \colon M_1 \to M_2$ that preserves the orientation classes, induces a map $f \colon H_\bullet(LM_1) \to H_\bullet(LM_2)$ that intertwines all the above operations, i.e. $(\star, \Delta, B)$. The operator $B$ is clearly homotopy-invariant. For the string product $\star$, it is shown in \cite{CohenKleinSullivan, crabb2008loop, gruher2008generalized} (or could be deduced from \cite{CohenJones02}) that it is homotopy-invariant. We show in this short note that this is \emph{not} true for the string coproduct. To that extent, we compute enough string coproducts on lens spaces to show that it is sensitive to Reidemeister torsion and transforms with respect to the Whitehead torsion. In particular, string topology can tell $L(1,7)$ and $L(2,7)$ apart. Moreover, we verify (in a certain range) the transformation formula
\begin{equation}
\label{eqn:mainconj}
\Delta f(x) = f(\Delta(x)) + f(x \star d\log\tau(f) ),
\end{equation}
where $\Delta$ is the string coproduct, $\star$ is the string product and $\tau(f)$ is the Whitehead torsion under the Dennis trace map, which we denote by $d\log$. Naturally, one is led to conjecture that this formula is true in full generality (i.e. for all closed manifolds $M$ and all $f \in \pi_*(\operatorname{aut}(M))$). Such a transformation formula is not entirely unexpected considering the following. In \cite{NaefWillwacher} it is shown that the natural comparison map (over the reals) between loop space cohomology and Hochschild homology of the cochain algebra $C^*(M)$ can be made to intertwine coproducts. The description of the coproduct on the algebraic side, however depends on the 1-loop contributions of the partition function of a Chern-Simons type field theory. It is moreover easy to see that not every $\mathrm{Com}_\infty$-automorphism of $C^*(M)$ (the algebraic analogue of a homotopy equivalence) preserves the coproduct since it might change the 1-loop part. In particular, the algebraic analogue of the above transformation formula is true, where the Whitehead torsion term is defined as the action on the 1-loop part. Stretching the analogy a bit we would like to think that this 1-loop part merely computes (a certain expansion of) the Reidemeister torsion as in the cellular model in \cite{cattaneo2020cellular}.

The structure of the paper is as follows. First we compute the integral homology of the free loop space of a lens spaces $M = L(k,7)$ and give generators. We proceed to compute all the string coproducts of generators in $H_3(LM)$ in terms of these generators. We then show that after quotienting out certain "inconvenient"
classes we can write particularly succinct formulas for the previous calculation, and that even after "forgetting" these classes, we can still detect Reidemeister torsion and get the correction terms as in \eqref{eqn:mainconj}. Finally, we show for one particularly striking example that the transformation formula \eqref{eqn:mainconj} is also true with the "inconvenient" classes intact.

\subsection*{Acknowledgements}
I would like to thank Pavel Mnev and Konstantin Wernli, the discussion with whom inspired this note, and Pavel Safronov for helpful discussions. I also would like to thank Nathalie Wahl for discussions, suggestions and support.

\tableofcontents

\section{Lens spaces}
In the following $M$ will be a lens space of the form $L(k,7)$. That is, let $t = e^\frac{1}{7}$ where $e = \mathrm{e}^{2\pi \sqrt{-1}}$ and consider the $\Z_7$-action on $S^3 = \{ (z_1, z_2) \ | \ |z_1|^2 + |z_2|^2 = 1 \}$ generated by
$$
(z_1, z_2) \mapsto (e^\frac{1}{7}z_1, e^\frac{k}{7}z_2).
$$
There is a residual two-torus action given by
$$
(z_1,z_2) \mapsto (e^{t} z_1, e^{kt+s} z_2),
$$
for $(s,t) \in [0,1] \times [0, \frac{1}{7}]$. Note that this action is free away from the two circles
$$
K_1 = \{ z_1 = 0\}
$$
and
$$
K_2 = \{ z_2 = 0\}.
$$
Let $r$ denote the inverse of $k \in \Z_7^{\times}$.

\subsection{Homology}
The free loop space $LM$ decomposes into $7$ connected components corresponding to elements in $\Z_7$. Let $L_l M$ denote the component corresponding to $l$. There is a fibration
\begin{equation}
\label{fibr}
\Omega_l M \to L_l M \to M,
\end{equation}
where $\Omega_l M$ is the component of the based loop space corresponding to $l$. Since $\Omega M$ is group-like, the component $\Omega_l M$ is homotopy-equivalent to $\Omega_0 M$, that is the component of contractible loops. Comparing homotopy groups we see that the map
$$
\Omega S^3 \to \Omega_0 M
$$
is a homotopy-equivalence. Hence the (integral) homology of $\Omega_l M$ is given by
$$
H_*(\Omega_l M) =
\begin{cases}
\mathbb{Z} & \text{for $* = 0,2,4,\dots$} \\
0 & \text{else}.
\end{cases}
$$
The differential on the $E_2$-page of the Serre spectral sequence associated to the fibration \eqref{fibr} is zero by degree reasons.
The $E_3$-page is
$$
\begin{tikzcd}
&&&\\
\dots &&& \\
\Z & \Z_7 & 0 & \Z\\
0 & 0 & 0 & 0 \\
\Z & \Z_7 & 0 & \Z \ar[uulll]\\
0 & 0 & 0 & 0 \\
\Z & \Z_7 & 0 & \Z \ar[uulll]
\end{tikzcd},
$$
where the only possibly non-zero differentials are indicated.
From the residual torus action one can see that the map $H_*(L_l M) \to H_*(M)$ is onto (for the component of the contractible loop it is onto for any space). We will see this in more detail below. Hence the differential on the $E_3$-page also vanishes and we obtain for the homology
\begin{align*}
    H_0(L_l M) &= \Z \\
    H_1(L_l M) &= \Z_7 \\
    H_2(L_l M) &= \Z \\
    H_3(L_l M) &= \Z \oplus \Z_7 \\
    H_4(L_l M) &= \Z \\
    \dots.
\end{align*}

Let us be more precise about these identifications and give more explicit descriptions for $H_0$, $H_1$ and $H_3$. We identify $H_0(LM)$ with $\Z[\Z_7]$. Moreover, let $\Omega^1 = \Omega^1(\Z[\Z_7]) = \F[t]/(t^7-1) \frac{dt}{t}$ denote the vector space of formal de Rham 1-forms and identify
\begin{align*}
    \F[t]/(t^7-1) \frac{dt}{t} &\longrightarrow H_1(LM) = \oplus_{l \in \Z_7} H_1(L_l M) \cong \oplus_{l \in \Z_7} H_1(M) \\
    ( c_0 + c_1 t + \dots c_6 t^6 ) \frac{dt}{t} &\longmapsto (c_0, \dots, c_6).
\end{align*}
\begin{Rem}
The reason we identify $H_1$ with $\Omega_1$ and not with $\Z[\Z_7]$ is first and foremost to make the formulas later more appealing. One can justify this identification at this point by appealing to the fact that $H_1(LM) = HH_1(\Z[\Z_7]) = \Omega_1$, where $HH_1$ is Hochschild homology. Thus the identification is in particular natural with respect to automorphisms of $\pi_1 = \Z_7$. Furthermore, the circle action $H_0(LM) \to H_1(LM)$ can now be written as the "de Rham differential"
\begin{align*}
\Z[\Z_7] &\to \Omega^1(\Z[\Z_7]) \\
c_0 + c_1 t + \dots + c_6 t^6 &\mapsto c_1 t + 2 c_2 t^2 + \dots + 6 c_6 t^6 \frac{dt}{t}
\end{align*}
\end{Rem}
Let us also define $\bar{\Omega}^1 = \Omega^1 / \Z \frac{dt}{t}$, such that we can identify
$$
\bar{\Omega}^1 \cong H_1(LM, M).
$$
Similarly $\bar{\Z}[\Z_7] = \Z [ \Z_7] / \Z 1$ such that
$$
\bar{\Z}[\Z_7] \cong H_0(LM,M).
$$
For $H_3(L_l M)$, the spectral sequence gives us a short exact sequence
$$
0 \to H_1(M, H_2(\Omega_l M)) \to H_3(L_l M) \to H_3(M) \to 0.
$$
From the residual two-torus action, we can construct a number of classes in $H_3(L_l M)$ that map to the fundamental class in $H_3(M)$. Consider the following $S^1$-actions. For given integers $(l,m)$ we define
$$
(t, z_1, z_2) \mapsto \rho_{l,m}(t, z_1, z_2) :=  (e^{lt}z_1, e^{(kl+7m)t}z_2) \quad \text{for $t \in [0,\frac{1}{7} ]$}.
$$
We view $\rho_{l,m}$ as a map $M \to LM$ and denote the image of the fundamental class by $[\rho_{l,m}] \in H_3(LM)$. The class of $\rho_{l,m}$ lies in the component corresponding to $l$,
$$
[\rho_{l,m}] \in H_3(L_lM),
$$
for all $m$, as can be seen for instance by setting $z_2 = 0$. We will argue below that these classes span all of $H_3(LM)$.


\subsection{String coproduct}
\subsubsection{Definition of the string coproduct}
Let us briefly recall how the string coproduct is computed. Given a homology class $\alpha \in H_p(LM)$ let us for simplicity assume that it is represented by a map $N \to LM$, where $N$ is an oriented $p$-dimensional manifold. In particular, we are given a map
\begin{align*}
    \alpha \colon S^1 \times N &\longrightarrow M \\
    (t,n) &\longmapsto \alpha(t,n).
\end{align*}
The self-intersection locus is defined by
$$
V = \{ (t,n) \ | \ \alpha(t,n) = \alpha(0,n) \ , \ t \neq 0 \} \subset S^1 \times N,
$$
and is (under appropriate transversality assumptions) an oriented manifold of dimension $p + 1 - \dim(M)$. Splitting the loops at intersection points we obtain
\begin{align*}
    \Delta(\alpha) \colon V &\longrightarrow LM \times LM \\
    (t,n) &\longmapsto ( s \mapsto \alpha(st,n) , s \mapsto \alpha( t+ (1-t)s, n) ).
\end{align*}
The resulting class is only well defined in $H_{p-n+1}( LM/M \times LM/M ) := H_{p-n+1}( (LM,M) \times (LM,M) ) := H_{p-n+1}( LM \times LM, LM \times M \cup M \times LM )$ and is the string coproduct of the class $\alpha$. One can furthermore show that it only depends on the image of $\alpha$ in $H_p(LM,M)$.
\subsubsection{String coproduct on $L(k,7)$}
In the case of $M = L(k,7)$ the string coproduct is a map $H_*(LM) \to H_{*+1-3}(LM/M \times LM/M)$.
We consider only the component
$$
\Delta\colon H_3(LM) \to H_{3+1-3}(LM/M \times LM/M) \to H_1(LM,M) \otimes H_0(LM,M).
$$

Under the identifications from the previous section the string coproduct gives a map
$$
H_3(LM) \to \bar{\Omega}^1 \otimes \bar{\Z}[\Z_7] = \F[t, t_2]\frac{dt}{t} /\left((t^7-1, t_2^7-1) \oplus \Omega^1 \cdot 1 \oplus \frac{dt}{t} \Z[t_2] \right) .
$$
More concretely, we identify a monomial $t^p t_2^q \frac{dt}{t}$ with the class in $H_1(L_p M) \otimes H_0(L_q M)$ whose image under $H_1(L_p M) \otimes H_0(L_q M) \to H_1(M) \otimes H_0(L_q M) \cong \Z_7$ is the canonical generator.


We now compute the coproduct of classes $[\rho_{l,m}]$ for $(l,m)$ positive coprime integers.

To compute the coproduct we have to find $(t,z_1,z_2)$ where
$$
\rho_{l,m}(t,z_1,z_2) = (z_1,z_2).
$$
Away from the circles $K_1 = \{ z_1 = 0 \}$ and $K_2 = \{ z_2 = 0 \}$ the action is free, since in that case the above equation reads as
$$
\begin{pmatrix}
l \\ kl +7m 
\end{pmatrix} t \in \Z^2 + \frac{1}{7}\begin{pmatrix}
1 \\ k
\end{pmatrix}\Z = \begin{pmatrix}
0 \\ 1
\end{pmatrix}\Z \oplus \begin{pmatrix}
1 \\ \frac{k}{7}
\end{pmatrix} \Z,
$$
whose solutions are $\frac{1}{7}\Z$ since $l$ and $m$ were coprime. Hence we 
only need to consider the self-intersection loci on the circles $K_1 = \{ z_1 = 0 \}$ and $K_2 = \{ z_2 = 0 \}$.
For $K_2 = \{ z_2 = 0 \}$ we chose the (orientation-preserving) coordinates $(\alpha, z)$ around $K_2$ via the assignment
$$
(\alpha, z) \mapsto (e^\alpha \sqrt{1 - |z|^2}, z).
$$
In these coordinates the action reads as
$$
\rho_{l,m}^c(t, \alpha, z) = (\alpha + lt, e^{kl+7m}t z).
$$
We obtain that the self-intersection locus is
$$
\{ (t,\alpha,z) \ | \ \rho_{l,m}^c(t, \alpha, z) - \rho_{l,m}^c(0, \alpha, z) = 0 \} =
\{ (t,\alpha,0) \ | \ t= \frac{1}{7l}, \frac{2}{7l}, \dots, \frac{l-1}{7l} \},
$$
that is a disjoint union of circles. The derivative of $\rho_{l,m}^c(t, \alpha, z) - \rho_{l,m}^c(0, \alpha, z)$ at $(\frac{n}{7l}, \alpha, 0)$ (computed on its cover) is given by
$$
\begin{pmatrix}
l & 0 & 0\\
0 & 0 & 0\\
0 & 0 & e^{(kl+7m)\frac{n}{7l}} - e^\frac{kn}{7}
\end{pmatrix},
$$
from which we see that if $l > 0$ then all the circles are oriented such that $s \to (e^s,0)$ is an orientation-preserving map. Each component of the self-intersection locus gives a term in $H_1(LM) \otimes H_0(LM)$. To identify these terms we only need to know which connected component it belongs to and what the image under $H_1(LM) \to H_1(M)$ is. The term belonging to $(\frac{n}{7l}, \alpha, 0)$ lies in the connected component associated to $t^n t_2^{l-n}$ and as we saw, the coefficient is given by the element in $H_1(M)$ that corresponds to $s \mapsto (e^s,0)$ which is the generator. Thus we get that the contribution from $K_2$ is
\[
(t t_2^{l -1} + t^2 t_2^{l -2} + \dots + t^{l-1} t_2) \frac{dt}{t}.
\]
Similarly, we obtain the contribution coming from $K_1 = \{ z_1 = 0 \}$, here the equation to solve is
$$
(0,e^{(kl+7m)t}z) = (0,z)
$$
intersection-locus is thus
$$
\{ (t,0,z_2) \ | \ t=\frac{1}{7(kl+7m)}, \frac{2}{7(kl+7m)}, \dots, \frac{kl+7m-1}{7(kl+7m)} \}.
$$
The contributions can again be expressed in terms of the class $\alpha \to (0, e^\alpha)$ which in $H_1(M)$ corresponds to $r$ where $r$ is the multiplicative inverse of $k$ mod $7$. Thus the contribution is
$$
r ( t^r t_2^{(kl+7m -1)r} t^{2r} t_2^{(kl+7m -2)r} + \dots + t^{(kl+7m -1)r} t_2^r) \frac{dt}{t}.
$$
Finally, we conclude that for any $(l,m)$ coprime
\begin{align}
\label{eqn:rhocop}
\begin{split}
\Delta([\rho_{l,m}]) =& (t t_2^{l -1} + t^2 t_2^{l -2} + \dots + t^{l-1} t_2) \frac{dt}{t} \\
&+ r ( t^r t_2^{(kl+7m -1)r} + t^{2r} t_2^{(kl+7m -2)r} + \dots + t^{(kl+7m -1)r} t_2^r) \frac{dt}{t}
\end{split}
\end{align}

It thus follows that
$$
\Delta([\rho_{l,m+n}] - [\rho_{l,m}]) = r \, n \, t^{kl+7m} (t \, t_2^6 + \dots + t^6 t_2 + t^7) \frac{dt}{t}.
$$
Recall that we are working in $H_1(LM,M) \otimes H_0(LM,M)$.
In particular, we see that we have found all lifts along $H_3(L_lM) \to H_3(M)$. More precisely, any two classes $[\rho_{l,m_1}]$ and $[\rho_{l,m_2}]$ for $m_1 \neq m_2 \text{ mod } 7$ and both coprime to $l$, are non-zero and not equal and hence span $H_3(L_lM)$. Alternatively, we actually see that the classes $[\rho_{l, 1 + nl}]$ for $n = 1, \dots, 7$ are all the lifts of the fundamental class along $H_3(L_lM) \to H_3(M)$ if $l \neq 0$. For $l = 0$ take the classes $[\rho_{7,n}]$ for $n = 1, \dots, 6$ and $[\rho_{0,0}]$.

\subsection{More convenient notation}
We wish to write to above formulas in a more convenient way. As we have seen $H_3(L_i M)$ is an extension of $H_3(M) = \Z$ by a $\Z_7$. As we have seen in the calculation above there is not much variation in the coproduct of the $\Z_7$ summand, so we are modding it out to simplify notation.
To that effect, let us denote the kernel of the map
$$
H_3(LM) \to \oplus_{i\in \Z_7} H_3(M)
$$
by $K$. Thus we can identify $H_3(LM)$ with $\Z[\Z_7]$ and our the formulas define a map
$$
H_3(LM)/K \to H_1(LM,M) \otimes H_0(LM,M) / \Delta(K),
$$
which we identify with
$$
\Z[\Z_7] \to \bar{\Omega}^1 \otimes \bar{\Z}[\Z_7] / \Delta(K).
$$
Our formulas actually lift to a map
$$
\Z[\Z_7] \to \Omega^1 \otimes \Z[\Z_7] / \Delta(K),
$$
which we will describe and at the very end project to $\bar{\Omega}^1 \otimes \bar{\Z}[\Z_7] / \Delta(K)$.
The target can be identified with the quotient of $\F[t,t_2, \frac{dt}{t}]$ by the subvector space spanned by
\begin{gather*}
t^l \frac{dt}{t}, t_2^l \frac{dt}{t}, t^l (t \, t_2^6 + \dots + t^6 t_2 + t^7) \frac{dt}{t}
\end{gather*}
Call that quotient $Q_{\F}$.

\subsection{Relation to Reidemeister and Whitehead torsion}
To rewrite the formulas in a more convenient way let us introduce a rational version of above target space. Introduce the algebra $\mathbb{Q}[t,t_2]$ with relations $t^7=1$, $t_2^7=1$. Consider the ideal $I = (t_2^6 + \dots + t^7)$ and define $A = \mathbb{Q}[t,t_2] / I$. This has the convenient effect that now $t^l - t_2^l$ are units in $A$, since $(t-t_2)(t^6 + 2 t^5 t_2 + \dots + t_2^7) = 7 + (t \, t_2^6 + \dots + t^6 t_2 + t^7)$. Let then $Q_{\mathbb{Q}}$ denote the vector space obtained by taking the quotient of $A$ by the subvector space spanned by the elements $t^l$ and $t_2^l$ and formally adjoint a symbol $\frac{dt}{t}$. Similarly there is an integral version of said space $Q_{\Z} \subset Q_{\mathbb{Q}}$. The above formulas define a map
$$
\Delta: \Z[\Z_7] \to Q_{\Z} \subset Q_{\mathbb{Q}}.
$$
After reduction mod $7$ this the component of the string corpoduct (the rationalization is merely to write down the formulas in a more conveninet way). The map $\Delta: \Z[\Z_7] \to Q_{\Z} \subset Q_{\mathbb{Q}}$ is defined by formulas
\begin{align*}
t^l \mapsto &(t t_2^{l -1} + t^2 t_2^{l -2} + \dots + t^{l-1} t_2) \frac{dt}{t} \\
&+ r ( t^r t_2^{(kl+7m -1)r} + t^{2r} t_2^{(kl+7m -2)r} + \dots + t^{(kl+7m -1)r} t_2^r) \frac{dt}{t} \\
=&(t t_2^{l -1} + t^2 t_2^{l -2} + \dots + t^{l-1} t_2 +t^l) \frac{dt}{t} \\
&+ r ( t^r t_2^{(kl+7m -1)r} + t^{2r} t_2^{(kl+7m -2)r} + \dots + t^{(kl+7m -1)r} t_2^r + t^{(kl+7m)r}) \frac{dt}{t}\\
=&(t_2^{l -1} + t t_2^{l -2} + \dots + t^{l-2} t_2 +t^{l-1}) dt \\
&+  ( t_2^{(kl+7m -1)r} + t^{r} t_2^{(kl+7m -2)r} + \dots + t^{(kl+7m -2)r} t_2^r + t^{(kl+7m-1)r}) dt^r \\
=& \frac{t^l - t_2^l}{t - t_2} dt + \frac{t^{r(kl+7m)} - t_2^{r(kl+7m)}}{t^r - t_2^r} dt^r \\
=& \frac{t^l - t_2^l}{t - t_2} dt + \frac{t^{l} - t_2^{l}}{t^r - t_2^r} dt^r \\
&= (t^l - t_2^l) d\log( (t^r-t_2^r)(t-t_2) ) \\
&= (t^l - t_2^l) d\log( R ),
\end{align*}

where $R \in Q$ is the homogenized Reidemeister torsion
$$
R = (t^r - t_2^r)(t-t_2).
$$
We refer to the lecture notes \cite[equation (58)]{MnevLecture} or \cite{milnor1966whitehead} for the fact that this is indeed the Reidemeister torsion (our convention differs slightly).
We summarize our findings in the following
\begin{Prop}
\label{prop:main1}
The string coproduct descends to
$$
\begin{tikzcd}
K \ar[r, "\Delta"] \ar[d] & \Delta(K) \ar[d] \\
H_3(LM) \ar[r, "\Delta"]\ar[d]  & H_1(LM,M) \otimes H_0(LM,M) \ar[d] \\
\oplus_{i\in \Z_7} H_3(M) \ar[r, "\mathcal{R}"] & \oplus_{i \neq 0 \in \Z_7} H_1(M) \otimes H_0(LM,M) / \Delta(K),
\end{tikzcd}
$$
where $\mathcal{R}$ is the map $t^l \mapsto (t^l - t_2^l) d\log( R )$ where $R$ is the homogenized Reidemeister torsion and the term $(t^l - t_2^l) d\log( R )$ is evaluated as explained above.
\end{Prop}

\begin{Rem}
For us Reidemeister torsion is merely an expression of the form $(t^p - t_2^p)(t^q - t_2^q)$. We do not fully explain here what the exact space of these expressions is. We will only need that the Whitehead group acts on these expressions faithfully.
\end{Rem}

\begin{Ex}
Let us give the calculation of $\mathcal{R}$ for $L(1,7)$ and $L(2,7)$
$$L(1,7), \quad k = 1, r = 1$$
\begin{align*}
t^0 \mapsto & 0\\
t^1 \mapsto & 0\\
t^2 \mapsto & 2 t^1 t_2^1 \frac{dt}{t}\\
t^3 \mapsto & 2 t^1 t_2^2 + 2 t^2 t_2^1 \frac{dt}{t}\\
t^4 \mapsto & 2 t^1 t_2^3 + 2 t^2 t_2^2 + 2 t^3 t_2^1 \frac{dt}{t}\\
t^5 \mapsto & 2 t^1 t_2^4 + 2 t^2 t_2^3 + 2 t^3 t_2^2 + 2 t^4 t_2^1 \frac{dt}{t}\\
t^6 \mapsto & 2 t^1 t_2^5 + 2 t^2 t_2^4 + 2 t^3 t_2^3 + 2 t^4 t_2^2 + 2 t^5 t_2^1 \frac{dt}{t}\\
\end{align*}

$$L(2,7), \quad k = 2, r = 4$$
\begin{align*}
t^0 \mapsto & 0\\
t^1 \mapsto & 4 t^4 t_2^4 \frac{dt}{t}\\
t^2 \mapsto & 5 t^1 t_2^1 + 4 t^4 t_2^5 + 4 t^5 t_2^4 \frac{dt}{t}\\
t^3 \mapsto & 5 t^1 t_2^2 + 5 t^2 t_2^1 + 4 t^4 t_2^6 + 4 t^5 t_2^5 + 4 t^6 t_2^4 \frac{dt}{t}\\
t^4 \mapsto & 5 t^1 t_2^3 + 5 t^2 t_2^2 + 5 t^3 t_2^1 + 4 t^5 t_2^6 + 4 t^6 t_2^5 \frac{dt}{t}\\
t^5 \mapsto & 2 t^1 t_2^4 + 5 t^2 t_2^3 + 5 t^3 t_2^2 + 2 t^4 t_2^1 + 4 t^6 t_2^6 \frac{dt}{t}\\
t^6 \mapsto & 2 t^1 t_2^5 + 2 t^2 t_2^4 + 5 t^3 t_2^3 + 2 t^4 t_2^2 + 2 t^5 t_2^1 \frac{dt}{t}\\
\end{align*}
In particular, we see that they cannot possibly be isomorphic. In the first one there are two $i$'s such that $H_3(L_iM) \to H_1(LM/M) \otimes H_0(LM/M)$ has rank one (or rank zero after quotienting out $\Delta(K)$). In the second one there is only one such $i$. Since all the $H_3(L_i M)$ have images in different components we see that the ranks of the maps $H_3(L_i M) \to H_1(LM,M) \otimes H_0(LM, M)$ differ.
\end{Ex}

We summarize the result of the above example in
\begin{Prop}
The string coproduct coalgebras on $L(1,7)$ and $L(2,7)$ are non-isomorphic. More precisely, they are told apart by the dimension of the kernel of $\Delta \colon H_3(LM, M) \to H_1(LM/M \times LM/M)$. For $M = L(2,7)$ the coproduct is injective on $H_3(LM, M)$, while for $M = L(1,7)$ the kernel is spanned by the class $[\rho_{1,0}]$.
\end{Prop}

\subsubsection{Whithead torsion}
Let $f: L(1,7) \to L(2,7)$ be a homotopy equivalence. Let $\tau(f) \in Wh(\Z_7) = (\Z [\Z_7])^\times / \Z_7$ be its Whitehead torsion. We denote by the same symbol its image under the map
$$
Wh(\Z_7) \to HH_1(\Z[\Z_7])/ HH_1(\Z[\Z_7], \Z) = H_1(LM/M) \to H_1(LM/M \times LM/M) \overset{1 \times \sigma}{\longrightarrow} H_1(LM/M \times LM/M),
$$
where $\sigma\colon LM \to LM$ is given precomposing with the orientation-reversing diffeomorphism of $S^1$. Let us recall the definition of the Dennis trace map in our case. The Hochschild homology computes as $HH_1(\Z[\Z_7]) = \Omega^1$ and $HH_1(\Z[\Z_7], \Z) = \Z_7$. Under these identifications the Dennis trace map is then given by
\begin{align*}
    Wh(\Z_7) &\to HH_1(\Z[\Z_7])/ HH_1(\Z[\Z_7], \Z) \\
    \alpha &\mapsto \alpha^{-1} d\alpha = d\log \alpha.
\end{align*}

The calculation in the previous section partially verifies the formula
$$
\Delta f(x) = f(\Delta(x)) + f(x \star d\log \tau(f) ),
$$
where $\star$ is the string product (applied to both factors as a derivation).
Namely, recall that $R_{2,7} = f(R_{1,7} \tau(f))$ and that moreover $\tau(f) \in \Z[\Z_7]^\times$. We then have
\begin{align*}
\Delta f(t^l)  &= (t^{f(l)} - t_2^{f(l)}) d\log R_{2,7} \\
&= (t^{f(l)} - t_2^{f(l)}) (f(d\log R_{1,7}) + d\log f(\tau(f))) \\
&= f(\Delta(t^l)) + (t^{f(l)} - t_2^{f(l)}) d\log f(\tau(f))) \\
&= f(\Delta(t^l)) + f((t^l - t_2^l) d\log \tau(f))),
\end{align*}
where we used the calculation of the string product from the Appendix.
\begin{Rem}
We used the following in the previous calculation. Let $I \to \F[\Z_7] \to \F$ be the augmentation ideal. Then $(t-1) \in I$ is not a zero-divisor in the algebra $I$, hence neither is $R_1 = R_{1,7} = (t-1)^2$ nor $R_2 = f(R_{2,7}) = (t^2-1)(t^4-1)$. Let $u \in \F[\Z_7]^\times$ be such that $R_1 = R_2 u$. Then to show the identity
$$
(t^l-1)d\log(R_1) - (t^l-1)d\log(R_2) = (t^1-1)du u^{-1} \quad \text{mod $\Sigma$},
$$
it is clearly enough to show that it is true after multiplying with $R_1 R_2 = R_2 R_2 u$. Doing this we obtain
\begin{align*}
    (t^l-1)(dR_1 R_2 - dR_2 R_2 u) &= (t^l-1)(R_2 R_2 du) \\
        &= R_2 R_2(t^l-1)du u^{-1}.
\end{align*}
\end{Rem}
Specializing to $l = 1$ we obtain
$$
\Delta f(t) = f( (t - t_2) d\log \tau(f) ).
$$
\begin{Ex}
It is known that there exists a homotopy equivalence $f\colon L(1,7) \to L(2,7)$ that sends the preferred generator $t$ to $t^2$. Its Whitehead torsion (in our convention) is thus
\begin{align*}
\tau(f) &= \frac{(t^4-1)(t^2 -1)}{(t-1)^2} = (t^3 + t^2 + t + 1) (t + 1) - \Sigma \\
&= t + t^2 + t^3 - t^5 - t^6 \\
\tau(f)^{-1} &= \frac{(t^8-1)(t^8-1)}{(t^4-1)(t^2-1)} = (1 + t^4)(1 + t^2 + t^4 + t^6) - \Sigma \\
&= t^4 - t^5 + t^6,
\end{align*}
where $\Sigma = 1 + t + t^2 + \dots + t^6$.
Its image under the Dennis trace is
\begin{align*}
d\log(\tau(f)) &= (1 + 2t + 3t^2 - 5 t^4 - 6 t^5)(t^4 - t^5 + t^6) dt \\
&= (6 +  5 t + 6 t^2 + t^3 + 2 t^4 + t^5 + 2 t^6) dt,
\end{align*}
and hence (after homogenizing again to match notation from above)
\begin{align*}
    (t - t_2) d\log \tau(f) &= (4 + 2 t t_2^6 + 6 t^3 t_2^4 + 2 t^5 t_2^2) dt - ( 1 + t t_2^6 + t^2 t_2^5 + t^3 t_2^4 + t^4 t_2^3 + t^5 t_2^2 + t^6 t_2) dt
\end{align*}
and finally
\begin{align*}
    f((t - t_2) d\log \tau(f)) &= (4 + 2 t^2 t_2^5 + 6 t^6 t_2 + 2 t^3 t_2^4) d(t^2) \\
    &= (4 + 2 t^2 t_2^5 + 6 t^6 t_2 + 2 t^3 t_2^4) 2 t dt \\
    &= (t^2 + 4 t^4 t_2^5 + 5 t t_2 + 4 t^5 t_2^4) \frac{dt}{t},
\end{align*}
where we dropped multiples of $\Sigma$.
\end{Ex}

Summarizing our findings we conclude with
\begin{Prop}
The string coproduct on the lens spaces $L(k,7)$ detects Whitehead torsion. More precisely, the restriction of the string coproduct to $H_3(LM)$ after taking the quotient described in Proposition \ref{prop:main1} transforms according to formula \eqref{eqn:mainconj} and two elements in $Wh(\Z_7)$ give the same correction term if and only if they are equal under the Dennis trace map.
\end{Prop}

\subsection{More details on an example}
Let us make the previous example more concrete and show that the formula in the introduction is still true even without modding out $K$ (i.e. dropping the multiples of $\Sigma$). To that extend recall that the homotopy equivalence $f$ is constructed as
$$
L(1,7) \to S^3 \vee L(1,7) \overset{\operatorname{id} \vee (z_1^2, z_2^4)}{\longrightarrow} S^3 \vee L(2,7) \overset{\Phi}{\longrightarrow} L(2,7),
$$
where $\Phi\colon S^3 \to L(2,7)$ is any map of degree $-7$ (see \cite[section 6.4]{MnevLecture}.
\begin{Lem}
$$
f([\rho_{1,0}]) = [\rho_{2,3}]
$$
\end{Lem}
\begin{proof}
Let us first try to compare the maps $f\circ \rho_{1,0} \colon S^1 \times L(1,7) \to L(2,7)$ and $\rho_{2,3} \circ (\operatorname{id} \times f) \colon S^1 \times L(1,7) \to L(2,7)$ using obstruction theory. Away from a neighborhood of a point in $L(1,7)$ the two maps are given by
\begin{align*}
f\circ \rho_{1,0} \colon (t, z_1, z_2) &\longmapsto (e^{2t} z_1^2, e^{4t} z_2^4 ) \\
\rho_{2,3} \circ (\operatorname{id} \times f) \colon (t, z_1, z_2) &\longmapsto (e^{2t} z_1^2, e^{(4 + 7 \cdot 3) t} z_2^4 )
\end{align*}
Recalling the standard cell decomposition of $L(1,7)$ as
\begin{align*}
    e_0 &= \{ (1,0) \} \\
    e_1 &= \{ (e^s,0) \ | \ s \in (0, \tfrac{1}{7}) \} \\
    e_2 &= \{ (z_1 ,r) \ | \ r \in (0, 1) \} \\
    e_3 &= \{ (z_1, z_2) \ | \ z_2 = e^s r \text{ for $s\in (0, \tfrac{1}{7})$} \},
\end{align*}
we see that the two maps already coincide on $A := \{ 0 \} \times L(1,7) \cup S^1 \times e_0$. The first obstruction for these two maps being homotopic relative to $A$ lies in
$$
H^3(M/A ; \pi_3( L(2,7) ) = \Z_7.
$$
The obstruction is computed by comparing the two maps on the $3$-cell $I \times e_2$. Since the maps coincide on the boundary of that cell, they fit together to a map $S^3 \to L(2,7)$, i.e. an element in $\pi_3(L(2,7)) = \Z$, where the identification is by computing the degree and dividing by $7$. Thus it is enough to show that the degrees of the two maps restricted to $I \times e_2$ are equal mod $49$. For the map $f\circ \rho_{1,0}$, we note that $\rho_{1,0}$ maps the cell $I \times e_2$ homeomorphically onto $e_3$. Since $f$ has degree 1, we get a contribution of $1$. For the map $\rho_{2,3} \circ (\operatorname{id} \times f)$, we note that $f$ is given by $(z_1^2, z_2^4)$ on the cell $I \times e_2$ and hence we are computing the degree of the map
\begin{align*}
    (0, \frac{1}{7}) \times e_2 &\longrightarrow L(2,7) \\
    (t, z_1, r) &\longmapsto (e^{2t} z_1^2, e^{(4 + 7 \cdot 3) t} r^4 ).
\end{align*}
This map has degree $2 \cdot 25 = 50$ (it has the same degree as its $7$-fold cover $S^1 \times e_2 \to S^3$ given by the same formula but now $t \in [0,1]$). We see that the obstruction vanishes since $1 \cong 50 \text{ mod } 49$. We conclude that the two maps in question are homotopic at least up to the $3$-skeleton of $S^1 \times L(1,7)$. They could still potentially differ on their $4$-cell $I \times e_3$ by an element in
$$
H^4(M/A ; \pi_4(L(2,7)) = \Z_2.
$$
We can view $\rho_{1,0}$ as an element in $\pi_1(\op{aut}_1(L(1,7))$, where $\op{aut}_1(L(1,7)$ is the monoid of self-equivalences homotopic to the identity. Under this identification, the action of $H^4(M/A ; \pi_4(L(2,7)) = \Z_2$ corresponds to multiplication by the element
\begin{align*}
&S^1 \times L(1,7) \to (S1 \times L(1,7)) \vee S^4 \to (S^1 \times L(1,7)) \vee S^3 \\
&\to L(1,7) \vee L(1,7) \to L(1,7) \in \pi_1(\op{aut}_1(L(1,7))),
\end{align*}
which is an element of order $2$. However, it follows directly from the definition of the string product that
$$
\pi_1(\op{aut}_1(L(1,7))) \to (H_3(LL(1,7)), \star)
$$
is a morphisms of monoids. Moreover, the image is contained in $\bigcup_{i \in \Z_7} H_3(LM)$ and maps to the fundamental class $[M]$ under $H_3(LM) \to H_3(M)$. We also saw that all these classes are of the form $\rho_{l,m}$. Thus we conclude that the image of $\pi_1(\op{aut}_1(L(1,7)))$ is a group of order $49$ and hence any element of order $2$ gets sent to zero. This shows that indeed
$$
f([\rho_{1,0}]) = [\rho_{2,3}] \in H_3(L L(2,7)),
$$
where $M = L(2,7)$.
\end{proof}
We are now ready to evaluate equation \eqref{eqn:mainconj} for $x = [\rho_{1,0}]$. The left hand side is given by
\begin{align*}
\Delta f([\rho_{1,0}]) =& \Delta [\rho_{2,3}] \\
=& t t_2 \frac{dt}{t} \\
& + 4 \cdot 3 ( t + t^2 t_2^6 + t^3 t_2^5 + t^4 t_2^4 + t^5 t_2^3 + t^6 t_2^2 + t_2) \frac{dt}{t} + 4 (t^4 t_2^5 + t t_2 + t^5 t_2^4 ) \frac{dt}{t} \\
=& 12 ( t + t^2 t_2^6 + t^3 t_2^5 + t^4 t_2^4 + t^5 t_2^3 + t^6 t_2^2 + t_2) \frac{dt}{t} + (4 t^4 t_2^5 +5 t t_2 +4 t^5 t_2^4 ) \frac{dt}{t}
\end{align*}
reading off formula \eqref{eqn:rhocop}.
For the right-hand side we obtain
$$
f(\Delta [\rho_{0,1}]) + f( [\rho_{1,0}] \star d\log\tau(f)) +  f( [\rho_{1,0}] \star d\log\tau(f)) =  f( (t - t_2) d\log\tau(f) ),
$$
using that $[\rho_{0,1}]$ has no self-intersections and the calculation of the string product in the appendix and introducing a Koszul sign. We already calculated the image of the Whitehead torsion under the Dennis trace map in the example above, that is
$$
d\log\tau(f) = (6 +  5 t + 6 t^2 + t^3 + 2 t^4 + t^5 + 2 t^6) dt \in HH_1(\Z[\Z_7]) = H_1(LM),
$$
the map $H_1(LM) \to H_1(LM \times LM) \to H_1(LM) \otimes H_0(LM)$ given by the diagonal and reversing the circle on the second factor sends a monomial $t^l \frac{dt}{t}$ to $t^l t_2^{-l} \frac{dt}{t}$ and hence is homogenization. We hence compute as in the above example (this time without dropping $\Sigma$ terms)
\begin{align*}
    f((t - t_2) d\log \tau(f)) &= (4 + 2 t^2 t_2^5 + 6 t^6 t_2 + 2 t^3 t_2^4) d(t^2) \quad + \quad 6 ( 1 + t t_2^6 + t^2 t_2^5 + t^3 t_2^4 + t^4 t_2^3 + t^5 t_2^2 + t^6 t_2) d(t^2) \\
    &= (4 + 2 t^2 t_2^5 + 6 t^6 t_2 + 2 t^3 t_2^4) 2 t dt \quad + \quad 12 ( 1 + t t_2^6 + t^2 t_2^5 + t^3 t_2^4 + t^4 t_2^3 + t^5 t_2^2 + t^6 t_2) dt\\
    &= (t^2 + 4 t^4 t_2^5 + 5 t t_2 + 4 t^5 t_2^4) \frac{dt}{t} \quad + \quad 12 ( t + t^2 t_2^6 + t^3 t_2^5 + t^4 t_2^4 + t^5 t_2^3 + t^6 t_2^2 + t_2) \frac{dt}{t}\\.
\end{align*}
Thus we see that the two sides of equation \eqref{eqn:mainconj} coincide up to the term $t^2 \frac{dt}{t}$ which corresponds to an element in $H_1(LM) \otimes H_0(M)$ and is hence zero in $H_1(LM, M) \otimes H_0(LM,M)$.

\section*{Appendix}
\subsection{String product}
We compute the string product 
$$
H_3(LM) \otimes H_1(LM) \to H_1(LM).
$$
Since the string coproduct is compatible with the projection to $M$, the only thing to check is which component we land in. It is then clear that
$$
[\rho_{l,m}] \star \omega \mapsto t^l \omega
$$
for any $\omega \in \Omega^1$.

\subsection{Transverse string topology}
We show that the transverse calculation of the string coproduct is indeed the invariantly defined string coproduct by comparing it with the definition in \cite{NaefWillwacher}. Recall that in loc. cit. the string coproduct is defined by the following zig-zag of spaces
\begin{equation}
\label{eqn:defredcop}
\begin{tikzcd}
\frac{LM}{M} \ar[dashed]{r}{\text{suspend}} &\frac{I \times LM}{\partial I \times LM\cup I\times M} \ar[r, "s"] & \frac{\Map(\bigcirc_2)}{F} \ar[r] & \frac{\Map(\bigcirc_2) / \Map^\prime(\bigcirc_2)}{ F/ F|_{UTM}} && \\
&&& \frac{\Map(8) / \Map^\prime(8)}{F / F|_{UTM}} \ar[u, "\simeq"] \ar[r, dashed, "Th"] & \frac{\Map(8)}{F} \ar[r] & \frac{LM \times LM}{LM \times M \cup M \times LM},
\end{tikzcd}
\end{equation}
where dashed arrows are only defined on homology and we used the following notations.
\begin{itemize}
    \item $UTM$ is the unit tangent bundle.
    \item $FM_2(M)$ is the compactified configuration space of two points, namely it is obtained from $M \times M$ by a real oriented blowup along the diagonal. It is a manifold with boundary $UTM$ and homotopy equivalent to $M \times M \setminus M$ and fits into the following commuting diagram
    \begin{equation*}
        \begin{tikzcd}
            UTM \ar[r] \ar[d] & FM_2(M) \ar[d] \\
            M \ar[r] & M \times M.
        \end{tikzcd}
    \end{equation*}
    \item $\Map(\bigcirc_2)$ is simply $LM$ thought of as a fibration over $M \times M$ given by evaluating the loop at time $0$ and $\frac{1}{2}$.
    \item Taking the pullback of $\Map(\bigcirc_2)$ along the above square we obtain
    \begin{equation*}
        \begin{tikzcd}
            \Map^\prime(8) \ar[r] \ar[d] & \Map^\prime(\bigcirc_2) \ar[d]\\
            \Map(8) \ar[r] & \Map(\bigcirc_2).
        \end{tikzcd}
    \end{equation*}    
    \item $F$ is $LM \sqcup LM \to M$.
    \item The map $s$ reparametrizes a loop. It takes a parameter $t \in I$ and a loop $\gamma$ and reparametrizes it in such a way that the path $\gamma_{[0,t]}$ is run through on the interval $[0, \frac{1}{2}]$ and the path $\gamma_{[t,1]}$ is run through on the inverval $[\frac{1}{2}, 1]$.
    \item The map $Th$ is capping with the Thom class in $H^n(M, UTM)$.
\end{itemize}

Let us first formulate the following
\begin{Lem}
Let $\alpha \colon N \to LM$ be transverse in the sense that
\begin{enumerate}
    \item [i)] $\frac{\partial}{\partial t} \alpha(u,t)$ is non-zero at $t=0$.
    \item [ii)] The map $\bar{\alpha} \colon N \times (0,1) \to M \times M$ given by $(n,t) \mapsto (\alpha(n,0), \alpha(n,t))$ intersects the diagonal transversely in a compact submanifold $V \subset N \times (0,1)$,
\end{enumerate}
then there is a unique map $\hat{\alpha} \colon \widehat{N \times I}^V \to FM_2(M)$ from the real oriented blowup of $N \times I$ at $V$ denoted by $\widehat{N \times I}^V$ to the compactified configuration space of two points such that
\begin{equation*}
    \begin{tikzcd}
    \widehat{N \times I}^V \ar[r] \ar[d] & FM_2(M) \ar[d] \\
    N \times I \ar[r] & M \times M
    \end{tikzcd}
\end{equation*}
commutes. Moreover, $\hat{\alpha}$ identifies the unit normal bundle of $V$ in $N \times (0,1)$ with $\bar{\alpha}|_V^* UTM$.
\end{Lem}
\begin{proof}
In local coordinates the map $M \times M \setminus M \to FM_2(M)$ looks like
\begin{align*}
    \R^n \times \R^n \setminus \R^n &\to \R^n \times S^{n-1} \times [0,\infty) \\
    (x,y) &\mapsto (x-y, \frac{x-y}{|x-y|}, |x-y| ).
\end{align*}
Composing with $\bar{\alpha}$ we readily see that condition i) is sufficient (and necessary) to lift the map $N \times I \to M \times M$ to $FM_2(M)$ in a neighborhood of $N \times \partial I$. To obtain the statement away from the boundary we observe that the function (in coordinates)
$$
(n,t) \to \frac{\bar{\alpha}(n,t) - \bar{\alpha}(n,0)}{|\bar{\alpha}(n,t) - \bar{\alpha}(n,0)|}
$$
smoothly extends from $N \times I \setminus V$ to $\widehat{N \times I}^V$.
\end{proof}

Such a transverse map $\alpha: N \to LM$ naturally defines a map
\begin{align*}
    \Delta(\alpha) \colon V &\to LM \times LM \\
    (n,t) &\mapsto (s \mapsto \alpha(n, st), s \mapsto \alpha(n, (1-t)s + t)).
\end{align*}

The following is a special case of Proposition 3.13 in \cite{HingstonWahl} adapted to our notation.
\begin{Prop}
Let $\alpha: N \to LM$ be transverse in the sense of the previous Lemma. Then
$$
\Delta(\alpha_*([N])) = (\Delta \alpha)_*([V]) \in H_*(LM \times LM, M \times LM \cup LM \times M)
$$
\end{Prop}
\begin{proof}
One checks that the following diagram commutes, where all the maps are the "obvious" ones.
\begin{equation*}
\begin{tikzcd}[row sep=scriptsize, column sep=scriptsize]
&\frac{N}{\varnothing} \ar[dashed]{r}{\text{suspend}} \ar[ddl] &\frac{I \times N}{\partial I \times N} \ar[r, equal] \ar[ddl] & \frac{I \times N}{\partial I \times N} \ar[r] \ar[ddl] & \frac{I \times N / \widehat{I \times N}^V}{\partial I \times N /\partial I \times N} \ar[ddl] & \\
&&&& \frac{V / \bar{\alpha}|_V^* UTM}{\varnothing / \varnothing} \ar[u, "\simeq"] \ar[r, dashed, "Th"] \ar[ddl] & \frac{V}{\varnothing} \ar[ddl] \ar[dd] \\ 
\frac{LM}{M} \ar[dashed]{r}{\text{suspend}} &\frac{I \times LM}{\partial I \times LM\cup I\times M} \ar[r, "s"] & \frac{\Map(\bigcirc_2)}{F} \ar[r] & \frac{\Map(\bigcirc_2) / \Map^\prime(\bigcirc_2)}{ F/ F|_{UTM}} && \\
&&& \frac{\Map(8) / \Map^\prime(8)}{F / F|_{UTM}} \ar[u, "\simeq"] \ar[r, dashed, "Th"] & \frac{\Map(8)}{F} \ar[r] & \frac{LM \times LM}{LM \times M \cup M \times LM},&
\end{tikzcd}
\end{equation*}
The only thing left to show is that after taking homology the upper zig-zag sends the fundamental class of $N$ to the fundamental class of $V$. Namely, we have to show that under
\begin{equation*}
    \begin{tikzcd}
    H_d(N) \ar[r] & H_{d+1}(I \times N, \partial I \times N) \ar[r] & H_{d+1}(I \times N, \widehat{I \times N}^V) && \\
    && \ar[u, "\simeq"] H_{d+1}(V, \bar{\alpha}|_V^* UTM) \ar[r, "Th"] & H_{d+1-n}(V)
    \end{tikzcd}
\end{equation*}
where $d = \op{dim}(N)$, the class $[N]$ gets sent to $[V]$. First note that the Thom isomorphism here is given by capping with a Thom class that is the pullback of the Thom class on $H_n(M ,UTM)$ along $\bar{\alpha}|_V$. This Thom class is also the natural Thom class by considering $\bar{\alpha}|_V^* UTM$ as the oriented normal bundle of $V$ in $I \times N$. Hence, apart from our insistence on avoiding tubular neighborhoods, we obtained the standard description of the intersection pairing from which it follows that $[N]$ is sent to $[V]$. To see this more concretely, we note that it is enough to show that composing with $H_{d+1-n}(V) \to H_{d+1-n}(V, V \setminus \{ x \} )$ sends $[N]$ to the generator in $H_{d+1-n}(V, V \setminus \{ x \} )$. Thus the situation is local and we can assume that $N = \R^d$ and $\bar{\alpha} \colon \R^{d+1} \to \R^n$ is a linear projection. In this case the statement follows directly from the definitions.
\end{proof}

\bibliographystyle{hplain}
\bibliography{main}

\end{document}